\documentclass[12pt,a4paper]{amsart}
\usepackage{amssymb}
\usepackage{amsmath}
\usepackage{amsthm}
\usepackage{mathrsfs}
\usepackage[hmargin=1.2in, vmargin=1.5in]{geometry}

\newtheorem{theorem}{Theorem}[section]
\newtheorem{proposition}{Proposition}[section]

\newtheorem{corollary}{Corollary}[section]

\theoremstyle{definition}
\newtheorem{definition}{Definition}[section]

\numberwithin{equation}{section}

\def\Z{\mathbb Z}
\theoremstyle{plain}

\newcommand{\R}{\mathbb{R}}
\newcommand{\N}{\mathbb{N}}

\newcommand{\flq}{\mathscr{P}l_{\omega}^q}
\newcommand{\flqone}{\mathscr{P}l_{\omega_1}^{q_1}}
\newcommand{\flqtwo}{\mathscr{P}l_{\omega_2}^{q_2}}

\newcommand{\flqloc}{\mathscr{P}l_{\omega, loc}^q}

\begin{document}

\title[Wave fronts via Fourier series coefficients]{Wave fronts via Fourier series coefficients}

\author[S. Maksimovi\'{c}]{Snjezana Maksimovi\'{c}}
\address{Faculty of Electrical Engineering, University of Banja Luka, Patre
5, 78000 Banja Luka, Bosnia and Herzegovina}
\email{snjezana.maksimovic@etfbl.net}

\author[S. Pilipiovi\'{c}]{Stevan Pilipovi\'{c}}
\thanks{This work was supported by the project \# 174024 of the Serbian Ministry of Education, Science and Technological Development.}
\address{Department of Mathematics and Informatics, University of Novi Sad, Trg Dositeja
Obradovica 4, 21000 Novi Sad, Serbia}
\email{pilipovic@dmi.uns.ac.rs}

\author[P. Sokoloski]{Petar Sokoloski}
\address{Faculty of Mathematics and Natural Sciences,  St. Cyril and Methodius University, Gazi Baba bb, 1000 Skopje, Macedonia}
\email{petar@pmf.ukim.mk}

\author[J. Vindas]{Jasson Vindas}
\thanks{J. Vindas gratefully acknowledges support by Ghent University, through the BOF-grant 01N01014.}
\address{Department of Mathematics, Ghent University, Krijgslaan 281 Gebouw S22,
9000 Gent, Belgium}

\email{jvindas@cage.ugent.be}

\begin{abstract}
Motivated by the product of periodic distributions, we give a new description of the wave front and the Sobolev-type wave front of a distribution $f\in\mathscr{D}'(\R^d)$ in terms of Fourier series coefficients.
\end{abstract}


\maketitle

\section{Introduction}

In this article microlocal properties of a distribution $f$ at $x_0\in
\R^d$ are detected through the Fourier series expansion of periodizations of $\varphi f$, where $\varphi$ is a cut-off function near $x_0$. In contrast to \cite{JPTTo}, where weighted type wave front sets have been discussed 
by the use of Gabor and dual Gabor frames depending on an additional continuous parameter $\varepsilon\rightarrow 0$, we shall show that the classical Fourier basis can be used for microlocal analysis. Our approach leads to discretized definitions of wave fronts in terms of Fourier coefficients.  This is the main novelty of the paper, which also includes proofs for the equivalence between  these discretized defintions and H\"{o}rmander's approach.

The space of periodic distributions is one of the basic Schwartz spaces and has been studied  in many books and papers in the second half of the last century. We refer here just to few of them \cite{Schwartz,Walter, Skornik, AMS, Beals, Kanwal} (see \cite{Vindas} for applications in summability of Fourier series). 
In the context of our paper, we mention the recent article \cite{RT} and book  \cite{Ruzh}, where Ruzhansky and Turunen have studied generalized functions on a torus $\mathbb T^d$. Their interest there lies in pseudo-differential operators and microlocal analysis over $\mathbb T^d\times\Z^d$. On the other hand, at present time, H\" ormander's notion of the wave front set attracts a lot of interest among mathematicians and there exists a vast literature related to this basic notion and its important role in the qualitative analysis of PDE and $\Psi$DO. We mention the basic books of H\" ormander \cite{Hor,Hor2} as standard references for classical and Sobolev type wave front sets; the articles \cite{pnt1,JPTTo} deal with weighted type wave fronts, while \cite{CJT1,CJT2} study extended wave fronts by considering local and global versions with respect to various Banach and Fr\' echet spaces of functions over the configuration and the frequency domains.

It is well known that the product of two distributions can be defined if their wave fronts are in a ``good'' position with respect to each other. This motivates us to 
study the product and wave fronts via spaces of  periodic distributions. In Subsection \ref{multiplication} below we discuss an elementary approach to local multiplication based on Fourier series. The main results of this article are presented in Sections \ref{wavefront} and \ref{Sobolev wavefront}. In Theorem \ref{wfth} and Theorem \ref{sobwf}
we characterize the wave front and the Sobolev type wave front of a distribution $f\in\mathscr{D}'(\R^d)$ by estimates of Fourier coefficients of its localizations. It should be mentioned that toroidal wave fronts have been studied in \cite{RT,Ruzh} through Fourier series as in this paper, 
but our approach is quite different and is related to the H\"{o}rmander's wave fronts in a precise fashion. Let us note that Sobolev type wave fronts were not considered in \cite{RT,Ruzh}.

\subsection{Notation}
For $x=(x_1,\ldots , x_d)\in\R^d$, we write $|x|=\sqrt{x_1^2+\ldots+x_d^2}$ and $\langle x\rangle=(1+|x|^2)^{1/2}$. Let $0<\eta\leq1.$ We will use the notation
 $$I_{\eta,x}=\prod_{j=1}^{d}(x_j-\eta/2,x_j+\eta/2) \: \: \mbox{and }\: \: I_\eta:= I_{\eta,0}.$$ 
Throughout the article, the word \emph{periodic} always refers to functions or distributions on $\mathbb{R}^{d}$ which are periodic of period 1 in each variable, i.e., $f(x+n)=f(x)$, $x\in\R^d, n\in\N^d$. We also use the notation $e_{y}$ for $e_y(x)=e^{2\pi i y \cdot x}$, $y\in\mathbb{R}^{d}$. We will consider periodic extensions of localizations of distributions around a point $x_0\in\mathbb{R}^{d}$, so if a distribution $g$ is supported by $I_{\eta,x_{0}}$, with $0<\eta<1$, we shall write $g_{p}(x):= \sum_{n\in\mathbb{Z}^{d}} g(x+n)$ for its periodic extension.

\subsection{Basic spaces}
The space of periodic test functions $\mathscr{P}=\mathscr{P}(\R^d)$ consists of
smooth periodic functions; its topology is given via the sequence of norms  $\|\varphi\|_k=\sup_{x\in I_1, |\alpha|\leq k}|\varphi^{(\alpha)}(x)|$,  $k\in\N$. Obviously,
$\varphi\in \mathscr{P}$ if and only if 
$\sum_{n\in\Z^d}\left|\varphi_n\right|^2 \langle n\rangle^{2k}<\infty$ for every   $k\in\Z,$
where $\varphi_n=\int_{I_1}\varphi(x)e^{-2\pi i n \cdot x}dx=\langle \varphi, e_{-n}\rangle, n\in\Z^d$.
The dual space of $\mathscr P$, the space of perioic distributions, is denoted by $\mathscr{P}^\prime$. One has:
$ f=\sum_{n\in\Z^d}{f_n e_n} \in \mathscr{P}^\prime$ if and only if $
\sum_{n\in\Z^d} \left| f_n \right|^2   \left\langle n\right\rangle^{-2k_0}<\infty,\ \textrm{ for some }k_0>0.$ If $f=\sum_{n\in\Z^d}{f_n e_n}\in\mathscr{P}'$ and $\varphi=\sum_{n\in\Z^d}{\varphi_n e_n}\in\mathscr{P}$,
 then their dual pairing is given by $\left\langle f,\varphi\right\rangle=\sum_{n\in\Z^d}f_n \varphi_{-n}$.

Let $\nu,\omega$ be positive functions over $\Z^d$. We call $\omega$ a $\nu$-moderate weight if there is $C$ such that
\begin{equation}\label{moderate}
\omega(m+n) \leq C   \omega(m)   \nu(n),  \: \: \forall m,n\in\Z^d. 
\end{equation}
If we take $\nu$ to be a polynomial, we call $\omega$ polynomially moderate. The set of all polynomially moderate weights on $\Z^d$ will be denoted as $\mathcal{P}ol(\Z^d)$. For $\omega \in \mathcal{P}ol(\Z^d)$, we define 
$$\flq=\{f\in \mathscr{P}' : \{f_n \omega (n)\}_{n\in\Z^d} \in l^q, \mbox{ where }f_n=\langle f,e_{-n}\rangle \}$$
supplied with the norm 
$\|f\|_{\flq}=\|\{f_n\omega(n)\}\|_{l^q}.$ We consider from now only values of $q\geq1$.
Clearly, $\flqone \subseteq \flqtwo$, if $q_1 \leq q_2$ and $\omega_2 \leq C   \omega_1$.

We will also consider the local space $\flqloc$ consisting of  distributions $f \in \mathscr{D}'(\R^d)$ such that the periodic extensions $(\varphi f)_{p}\in \flq$, for all $x_{0}\in\mathbb{R}^{d}$ and $\varphi \in \mathscr{D}(I_{1,x_0})$. Its topology is defined via the family of seminorms
$\|f\|_{x_0,\varphi}=\| (\varphi f)_{p}\|_{\flq}$, where $x_{0}\in\mathbb{R}^{d}$ and $  \varphi \in \mathscr{D}(I_{1,x_{0}}).$ For the sake of completeness, we give the following elementary proposition (its proof also shows that the definition of $\flqloc$ is consistent).

\begin{proposition}\label{prop1}
$\flq\subset {\mathscr{P}l_{\omega,loc}^q}$.
\end{proposition}
\begin{proof} Let $f\in \flq$ and $\varphi\in \mathscr{D}(I_{1,x_{0}})$, then $(\varphi f)_{p}=\varphi_{p}f$. Write $f=\sum_{n}{f_n e_n}$ and $\varphi_{p}=\sum_{n}{\varphi_n e_n}\in \mathscr{P}$. By \eqref{moderate} and the generalized Minkowski inequality, we have
\begin{align*}\|\varphi_{p}  f\|_{\flq}&\leq C \left( \sum_n \left( \sum_j |\varphi_{j}| \nu ( j ) | f_{n-j}|  \omega(n-j)\right)^q\right)^{1/q}
\\
&
\leq 
C\|\varphi_{p}\|_{\mathscr{P}l_{\nu}^1} \|f\|_{\flq}<\infty.
\end{align*}

\end{proof}

Set $\omega_s(n)=\left\langle n\right\rangle^s,$ $s\in \mathbb{R}$. For convenience, we write
$\mathscr{P}l_{s}^q:=\mathscr{P}l_{\omega_{s}}^q$ and
$\mathscr{P}l_{s, loc}^q:=\mathscr{P}l_{\omega_{s}, loc}^q$. We clearly have

$$\mathscr{P}=\bigcap_{s\geq 0}{\mathscr{P}l_s^q}= \bigcap _{\omega \in \mathscr P ol(\mathbb{Z}^{d})}\flq \mbox{ and  } \mathscr{P}'=\bigcup_{s\leq 0}{\mathscr{P}l_s^q}=\bigcup _{\omega \in \mathscr P ol(\mathbb{Z}^{d})}\flq\: .$$
Moreover,
$$\mathscr{E}=\bigcap_{s\geq 0}{\mathscr{P}l_{s,loc}^q}= \bigcap _{\omega \in \mathscr P ol(\mathbb{Z}^{d})}\flqloc \mbox{ and  } \mathscr{D}'_F=\bigcup_{s\leq 0}{\mathscr{P}l_{s,loc}^q}=\bigcup _{\omega \in \mathscr P ol(\mathbb{Z}^{d})}\flqloc \:,$$
where $\mathscr E$ is the space of all smooth functions and $\mathscr D'_F$ is the space 
of distributions of finite order on $\R^d.$

\subsection{Multiplication}\label{multiplication}
In this subsection we make some comments about the multiplication of distributions. Assume that the indices $q,q_1,q_2\in[1,\infty]$ are such that $\frac{1}{q_1}+\frac{1}{q_2}=\frac{1}{q}+1$. We fix two weight functions $\omega, \nu \in \mathcal P{ol}(\mathbb{Z}^{d})$ and we assume $\omega$ is $\nu$-moderate (cf. (\ref{moderate})). 

We start with products in the spaces of type $\flq$. Here we define the product via Fourier coefficients. Indeed, let  $f_1=\sum_{n\in\mathbb{Z}^{d}}f_{1,n}e_n\in\mathscr{P}l_{\omega}^{q_1}$ and $f_2=\sum_{n\in\mathbb{Z}^{d}}f_{2,n}e_n\in\mathscr{P}l_{\nu}^{q_2}.$ We define their  product as $f:=f_{1}f_{2}:= \sum_{n\in\mathbb{Z}^d}f_{n}e_{n}$, where 
$$f_n=\sum_{j\in\mathbb{Z}^{d}}f_{1,n-j}f_{2,j}, \:\: n\in\Z^d.$$ 
We will check in Proposition \ref{prop2} that $f\in\mathscr{P}l_{\omega}^{q}$.

The previous definition allows us to introduce multiplication in the local versions of these spaces. In fact, let now $f_1\in\mathscr{P}l_{\omega,loc}^{q_1}$ and $f_{2}\in\mathscr{P}l_{\nu,loc}^{q_2}$. To define their product $f=f_{1}f_{2}$, we proceed locally. Let $x_{0}\in\mathbb{R}^{d}$ and $0<\eta<1$. Let $\phi\in \mathscr{D}(I_{1,x_{0}})$ be such that $\phi(x)=1$ for $x\in I_{\eta,x_{0}}$. We define $f_{I_{\eta,x_{0}}}\in\mathscr{D}'(I_{\eta,x_{0}})$ as the restriction of  $(\phi f_{1})_{p} (\phi f_{2})_{p}$ to $I_{\eta,x_{0}}$. Note that different choices of $\phi$ lead to different Fourier coefficients but, by Proposition \ref{prop1}, we have $f_{I_{\eta,x_{0}}}=f_{I_{\eta',x'_{0}}}$ on  $I_{\eta,x_{0}}\cap I_{\eta',x'_{0}}$. The $\{f_{I_{\eta,x_{0}}}\}$ thus gives rise to a distribution $f\in\mathscr{P}l_{\omega,loc}^{q}$ and we define the product of $f_{1}f_{2}:=f$.

    \begin{proposition}\label{prop2} The mappings 
\begin{equation}\label{mapping1}
\mathscr{P}l_{\omega}^{q_1}\times\mathscr{P}l_{\nu}^{q_2}\ni (f_{1},f_{2})\mapsto f_{1}f_{2} \in \mathscr{P}l_{\omega}^{q}
\end{equation}
and
\begin{equation}\label{mapping2}
\mathscr{P}l_{\omega,loc}^{q_1}\times\mathscr{P}l_{\nu,loc}^{q_2}\ni (f_{1},f_{2})\mapsto f_{1}f_{2} \in\mathscr{P}l_{\omega,loc}^{q}
\end{equation}
are continuous. 
\end{proposition}
\begin{proof}
The continuity of (\ref{mapping2}) follows at once from that of (\ref{mapping1}). For (\ref{mapping1}), Young's inequality and (\ref{moderate}) yield 
$$\|f_{1}f_{2}\|_{\flq}\leq  C \|f_{1}\|_{\mathscr{P}l_{\omega}^{q_1}} \|f_{2}\|_{\mathscr{P}l_{\nu}^{q_2}}\:.$$
\end{proof}
In particular, we have:
\begin{corollary}\label{cor1} Let  $s, s_1, s_2\in \R$ be such that 
\begin{equation}
s_1+s_2\geq0,\qquad s\leq\min\{s_1,s_2\}.
\end{equation}
Then, the two mappings $\mathscr{P}l_{s_{1}}^{q_1}\times\mathscr{P}l_{s_{2}}^{q_2}\ni (f_{1},f_{2})\mapsto f_{1}f_{2} \in \mathscr{P}l_{s}^{q}$ and $\mathscr{P}l_{s_{1},loc}^{q_1}\times\mathscr{P}l_{s_{2},loc}^{q_2}\ni (f_{1},f_{2})\mapsto f_{1}f_{2} \in \mathscr{P}l_{s,loc}^{q}$ are continuous.           
\end{corollary}
\begin{proof}
We may assume $s_{1}\geq 0$ and $s=s_2$. It is obvious that $s_1\geq\left|s_2\right|$ has to hold in order to have $s_1+s_2\geq0$. The result then follows from Proposition \ref{prop2} upon setting $\omega(n)=\left\langle n\right\rangle^{s_2}$ and 
 $\nu(n)=\left\langle n\right\rangle^{s_1}$ because (\ref{moderate}) holds for them.
\end{proof}

Concerning the local products from Corollary \ref{cor1}, exactly the same method from the proof of Theorem \ref{sobwf} below applies to show that the local space $\mathscr{P}l_{s,loc}^{2}$ coincides with the local Sobolev space $H^{s}_{loc}(\mathbb{R}^{d})$. Therefore, the multiplicative product for the local spaces in Corollary \ref{cor1} agrees with the one defined by H\"{o}rmander in \cite[Sect. 8.2]{Hor2}. Moreover, it is also worth mentioning that our results from the next sections imply that one can go beyond local products and in fact define the multiplicative product by \emph{microlocalization} as in \cite[Sect. 8.3]{Hor2}. We leave the formulation of such definitions to the reader. Theorem \ref{sobwf} below shows that the microlocal version of our multiplication also agrees with H\"{o}rmander's one.

\section{Wave Front}\label{wavefront}

 Our goal in this section is to describe the wave front of $f\in\mathscr{D}^\prime(\R^d)$ via the Fourier series coefficients of the periodic extension of an appropriate localization of 
 $f$ around $x_0\in\R^d$, as explained in the previous section. Recall $(x_0,\xi_0)\notin WF(f)$ 
 if there exist  $\psi\in \mathscr D(\mathbb{R}^{d})$ with $\psi\equiv 1$ in a neighborhood of $x_{0}$ and an open cone $\Gamma\subset\mathbb{R}^{d}$ containing $\xi_0$ such that
\begin{equation}\label{WF}
(\forall N>0)(\exists C_N>0)(\forall \xi\in\Gamma)(|\widehat{f\psi}(\xi)|\leq C_N\langle \xi\rangle^{-N}). 
\end{equation}
The next theorem tells that we can discretize (\ref{WF}):
\begin{theorem}\label{wfth}
Let $f\in\mathscr{D}'(\R^d)$ and $(x_{0},\xi_0)\in\mathbb{R}^{n}\times(\R^d\setminus\{0\})$. The following conditions are equivalent:

$(i)$  There exist  
		 $\phi\in \mathscr D(I_{\varepsilon,x_0})$, with $\varepsilon\in(0,1)$ and $\phi\equiv 1$ in a neighborhood of $x_0$, and an open cone $\Gamma$ containing $\xi_{0}$ such that
\begin{equation}\label{A}
(\forall N\in\N)(\exists C_N>0)(\forall n\in\Gamma\cap\Z^d)
(|\widehat{f\phi}(n)|\leq C_N \langle n\rangle^{-N}).\end{equation}

$(ii)$  $(x_0,\xi_0)\notin WF(f)$.
\end{theorem}
\begin{proof}
 It is well known that, by shrinking the conic neighborhood of $\xi_{0}$, one may choose $\psi$ in (\ref{WF}) with arbitrarily small support around $x_{0}$. Thus, $(ii)$ implies $(i)$. So, it is enough to show that $(i)$ implies $(ii)$. Assume $(i)$. We divide the proof in two steps. We first prove that there are $\varepsilon'$ and an open cone $\xi_{0}\in\Gamma_{1}$ such that
$$
(\forall B \mbox{ bounded set in } \mathscr{D}(I_{\varepsilon',x_{0}}))(\forall N>0)(\exists C'_N>0)
$$
\begin{equation}\label{WFcond}(\forall n\in\Gamma_1\cap \mathbb{Z}^{d})\left(\sup_{\varphi\in B}|\widehat{f\varphi}(n)|\leq\frac{C'_N}{\langle n\rangle^N}\right).
\end{equation}
We choose $\varepsilon'$ in such a way that $\phi(x)=1$ for every $x\in I_{\varepsilon',x_{0}}$. For the cone, we select $\Gamma_{1}$ an open cone with $\xi_{0}\in \Gamma_{1}$ and $\overline{\Gamma}_1\subset\Gamma\cup\left\{0\right\}$. Let us show that (\ref{WFcond}) holds with these choices. Let $0<c<1$ be a constant smaller than the distance between $\partial \Gamma$ and the intersection of $\overline{\Gamma}_{1}$ with the unit sphere. Clearly,  $\left\{y\in\mathbb{R}^{d}: (\exists\xi \in \Gamma_1)(|\xi-y|\leq c|\xi|)\right\}\subset \Gamma$.
Let $B\subset \mathscr D(I_{\varepsilon',x_{0}})$ be a bounded set. We have that 
 $\phi\varphi=\varphi,$ $\forall\varphi\in B$. Moreover, note that $\widehat{f\varphi}(n)$ are precisely the Fourier coefficients of the periodic distribution $(f \phi)_{p}(\varphi)_{p}$. 
Therefore, for $\varphi\in B$ and $n\in \Gamma_1\cap \mathbb{Z}^{d}$,
\begin{align*}
\left|\widehat{f\varphi}(n)\right|
=\left|\sum_{j\in\mathbb{Z}^{d}}\widehat{f\phi}(n-j)\widehat{\varphi}(j) \right|
&
\leq
\left(\sum_{|j|\leq c |n|}+\sum_{|j|> c |n|}\right)|\widehat{f\phi}(n-j)\widehat{\varphi}(j)|
\\
&
=:I_1(n)+I_2(n)
\end{align*}
Further on, 
$$I_1(n)=\sum_{| n-j | \leq c |n |}| \widehat{f \phi} (j) | | \widehat{\varphi} (n-j ) |\leq C \sup_{|n-j|\leq c|n|}| \widehat{f \phi} (j) |,$$
where $C$ only depens on $B$. Since $| n- j | \leq c | n |$ implies $| j | \geq(1-c)| n |,$ 
\begin{equation}\label{1}
\begin{split}
\sup_{\varphi\in B,\: n\in\Gamma_1}\langle n\rangle^N I_1(n) \leq C\sup_{n\in\Gamma_1} \langle n\rangle^N
\sup_{|n-j| \leq c |n | } {| \widehat{f\phi}(j)| } \\
\leq C \sup_{j\in\Gamma_{\xi_0}}{(1-c)^{-N}\langle j\rangle^N |\widehat{f\phi}(j)|}= C(1-c)^{-N}C_N.
\end{split}
\end{equation}
For the estimate of $I_2$ we use that $|n-j|\leq(1+c^{-1})|j|$ if $|j|\geq c |n|$. 
Moreover, by the Paley-Wiener theorem, there are $M, D>0$ such that
$$|\widehat{f\phi}(n-j)|\leq D\langle n-j\rangle^M,\ \ n,j\in\mathbb{Z}^d.$$
Due to the boundedness of $B\subset\mathscr D(\R^d)$,
$$\sup_{\varphi\in B}\sum_{j\in\mathbb{Z}^{d}}\langle j\rangle^{M+N}|\widehat{\varphi}(j)|=:K_{N}<\infty.$$
Thus, for the second term,  we have
\begin{equation}\label{*}
\begin{split}
\sup_{n\in\Gamma_1}\langle n\rangle^N I_2(n) \leq D
\sup_{n\in\Gamma_1}\langle n\rangle^N\sum_{| j |\geq c |n|} \langle n-j\rangle^M |\widehat{\varphi}(j)| \\
\leq Dc^{-N}(1+c^{-1})^M K_{N}, \ \ \forall \varphi\in B.
\end{split}
\end{equation}
Combining (\ref{1}) and (\ref{*}), we get (\ref{WFcond}).
 
We now deduce that $(x_{0},\xi_{0})\notin WF(f)$ with the aid of (\ref{WFcond}). Let $\psi \in\mathscr{D}(I_{\varepsilon',x_{0}})$ be equal to 1 in a neighborhood of $x_{0}$.
Then, the set $B=\{\varphi_{t}:=\psi e_{-t}:\: t\in [0,1)^{d}\}$ is a bounded subset of $\mathscr{D}(I_{\varepsilon',x_{0}})$. So,
$$\sup_{t\in [0,1)^{d}}|\widehat{f\psi}(n+t)|=\sup_{t\in [0,1)^{d}}|\widehat{f\varphi_{t}}(n)|\leq\frac{C'_{N}}{\langle n\rangle^{N}},\ \ \forall n\in\Gamma_1\cap \mathbb{Z}^{d}, $$
that is, 
\begin{equation}\label{br2}
\sup_{\xi\in (\Gamma_1\cap \mathbb{Z}^{d})+[0,1)^{d}}\langle \xi\rangle^{N}|\widehat{f\psi}(\xi)|\leq (1+4d)^{N/2}C'_{N}.
\end{equation}
Select now an open conic neighborhood $\Gamma_{2}$ of $\xi_{0}$ such that $\overline{\Gamma}_{2}\subset \Gamma_{1}\cup\{0\}$ and find $c'$ such that $\left\{y\in\mathbb{R}^{d}: (\exists\xi \in \Gamma_2)(|\xi-y|\leq c'|\xi|)\right\}\subset \Gamma_{1}$. The latter condition implies that $\Gamma_2\cap \{\xi\in\mathbb{R}^{d}:\: |\xi|c'\geq 1\}\subset (\Gamma_1\cap \mathbb{Z}^{d})+[0,1)^{d}$ and hence
$$
\sup_{\xi\in \Gamma_2}\langle \xi\rangle^{N}|\widehat{f\psi}(\xi)|\leq \max\{C''_{N},(1+4d)^{N/2}C'_{N}\}=C_{N}<\infty
,$$
where $C''_{N}=\sup_{\xi\in \Gamma_2,\: |\xi|<1/c'}\langle \xi\rangle^{N}|\widehat{f\psi}(\xi)|$. This shows that $(x_{0},\xi_{0})\notin WF(f)$, as claimed.
\end{proof}

We mention that Theorem \ref{wfth} also follows from the relation between discrete and H\"{o}rmander's wave front sets proved in \cite[Thm. 7.4]{RT}, once one observes that the notion is local and so it does not depend of a particular parametrization.

\section{Sobolev wave front}\label{Sobolev wavefront}
In this section we deal with wave fronts of Sobolev type. We slightly reformulate H\"ormander's definition \cite{Hor2}.
\begin{definition}
Let $f\in\mathscr{D}'(\mathbb{R}^{d})$, $(x_{0},\xi_0)\in\mathbb{R}^{d}\times(\R^d\setminus\{0\})$, and $s\in\mathbb{R}$. We say that $f$ is Sobolev microlocally 
regular at $(x_0,\xi_0)$ of order $s$, that is $(x_0,\xi_0)\notin WF_s(f)$, if there exist an open cone $\Gamma$ around $\xi_0$ and $\psi\in \mathscr D(\mathbb{R}^{d})$ with $\psi\equiv 1$ in a neighborhood of $x_{0}$ such that 
\begin{equation}\label{sob1}
\int_\Gamma{|\widehat{\psi f}(\xi)|^2\langle \xi\rangle^{2s} d\xi<\infty}.
\end{equation}
\end{definition}
We shall now refine Theorem \ref{wfth}:
\begin{theorem}\label{sobwf}
Let $f \in \mathscr{D}^\prime(\R^d)$. The following two conditions are equivalent:

$(i)$ There exist an open cone $\Gamma$ around $\xi_0$, $\phi\in \mathscr D(I_{\eta,x_{0}})$, $\eta\in(0,1),$ with $\phi\equiv 1$ in a neighborhood of $x_{0}$, such that 
\begin{equation}\label{sob2}
\sum_{n\in\Gamma\cap\Z^d}{|a_n|^2\langle n\rangle^{2s}}<\infty, \ \ \ \mbox{where }(f\phi)_{p}=\sum_{n\in\mathbb{Z}^{d}} a_n e_n.
\end{equation}

$(ii)$ $(x_0,\xi_0)\notin WF_s(f).$
\end{theorem}
\begin{proof}
$(i)\Rightarrow (ii)$. Assume (\ref{sob2}). Choose an open cone $\Gamma_1$ so that $\overline{\Gamma_1}\subset\Gamma\cup\left\{0\right\}$ and $\xi_{0}\in \Gamma_{1}$.  Find $0<\varepsilon<\eta$ such that $\phi(x)=1$ for all $x\in I_{\varepsilon,x_{0}}$. We first prove that: For every bounded set $B\subset\mathscr{D}(I_{\varepsilon,x_{0}})$ 
			\begin{equation}\label{sob3}
					\sup_{\varphi\in B}\sum_{n\in\Gamma_{1}\cap \mathbb{Z}^{d}}|\widehat{f\varphi}(n)|^2\langle n\rangle^{2s}<\infty.
			\end{equation}
Fix a bounded subset $B\subset\mathscr{D}(I_{\varepsilon,x_{0}})$. In view of the choice of $\varepsilon$, we have that $f\varphi=f\varphi\phi$ and  so
$\widehat{f\varphi}(n)=\sum_{j\in \mathbb{Z}^{d}}a_{j}\widehat{\varphi}(n-j),$ for every $\varphi\in B$.  We fix a constant $0<c<1$ that is smaller than the distance between $\partial \Gamma$ and the intersection of $\overline{\Gamma}_{1}$ with the unit sphere, and also smaller than the distance between $\partial \Gamma_{1}$ and the intersection of $\overline{\mathbb{R}^{d}\setminus\Gamma}$ with the unit sphere.  One has that  $\xi\in\Gamma_1$ and $y\notin\Gamma$ imply $|\xi-y|>c \max{\{|\xi|,|y|\}}$. We keep $\varphi\in B$. By Peetre's inequality, we have
\begin{align*}
\left(\sum_{n\in \Gamma_1\cap \mathbb{Z}^{d}}|\widehat{f\varphi}(n)|^2\langle n\rangle^{2s} \right)^{\frac{1}{2}}
& 
\leq C\left(\sum_{n\in \Gamma_1\cap \mathbb{Z}^{d}} \left(\sum_{j\in \mathbb{Z}^{d}}|a_{j}|\langle j\rangle^s\left|\widehat{\varphi}(n-j)|\langle n-j\right\rangle^{|s|} \right)^2 \right)^{1/2}
\\
&\leq  C(I_1(\varphi)+I_2(\varphi)),
\end{align*}
where
$$I_1(\varphi)=\left(\sum_{n\in \Gamma_1\cap \mathbb{Z}^{d}} \left(\sum_{j\in\Gamma\cap\mathbb{Z}^{d}}|a_{j}|\langle j\rangle^s\left|\widehat{\varphi}(n-j)|\langle n-j\right\rangle^{|s|} \right)^2 \right)^{1/2}$$
and
$$I_2(\varphi)=\left(\sum_{n\in \Gamma_1\cap \mathbb{Z}^{d}} \left(\sum_{j\notin\Gamma\cap \mathbb{Z}^{d}}|a_{j}|\langle j\rangle^s\left|\widehat{\varphi}(n-j)|\langle n-j\right\rangle^{|s|} \right)^2 \right)^{1/2}.$$
By Young's inequality and the fact that $B$ is a bounded set,
$$
\sup_{\varphi\in B}I_{1}(\varphi)\leq \left(\sum_{n\in\Gamma\cap\Z^d}{|a_n|^2\langle n\rangle^{2s}}\right)^{1/2}\sup_{\varphi\in B} \sum_{n\in \mathbb{Z}^{d}}|\varphi(n)|\langle n\rangle^{|s|}<\infty.
$$
We now estimate $I_2(\varphi)$. Since $f\phi$ is compactly supported,  
$\langle j \rangle^{s}|a_{j}|=\langle j \rangle^{s}|\widehat{f\phi}(j)|\leq D\langle j\rangle^k,$ $\forall n\in\mathbb{Z}^d,$ for some $D>0$ and $k>0$.
The fact that $B$ is bounded yields the existence of $C'>0$ such that 
$|\widehat{\varphi}(j)|\leq C'\langle j \rangle^{-k-|s|-3(d+1)/2}.$ Because of the choice of $\Gamma_1$, we have
\begin{align*}\sup_{\varphi\in B}(I_2(\varphi))^{2}&\leq (DC')^{2}\sum_{n\in \Gamma_1\cap \mathbb{Z}^{d}} \left(\sum_{j\notin\Gamma\cap \mathbb{Z}^{d}}\langle j\rangle^{k}\langle n-j\rangle^{-k-3(d+1)/2} \right)^2
\\
&
\leq 
(DC')^{2}c^{-2k-3(d+1)}\sum_{n\in \Gamma_1\cap \mathbb{Z}^{d}}\langle n\rangle^{-d-1}\left(\sum_{j\notin\Gamma\cap \mathbb{Z}^{d}}\langle j\rangle^{-d-1} \right)^2.
\end{align*}
Thus $(\ref{sob3})$ has been established. We now deduce $(ii)$ from (\ref{sob3}). Once again we shrink the conic neighborhood of $\xi_{0}$. So let $\Gamma_2$ be an open cone such that $\overline{\Gamma_2}\subset \Gamma_{1}\cup\left\{0\right\}$ and $\xi_{0}\in \Gamma_{2}$. Let $\psi \in\mathscr{D}(I_{\varepsilon,x_{0}})$ be equal to 1 in a neighborhood of $x_{0}$. Find $r>0$ such that $\Gamma_2\cap \{\xi\in\mathbb{R}^{d}:\: |\xi|\geq r\}\subset (\Gamma_1\cap \mathbb{Z}^{d})+[0,1)^{d}$.
For each $n\in \Gamma_{1}\cap \mathbb{Z}^{d}$, write $\Lambda_n= n+[0,1]^{d}$. Then, by (\ref{sob3}) and Peetre's inequality,
\begin{align*}
\underset{|\xi|\geq r}{\int_{\xi\in\Gamma_{2}} }|\widehat{\psi f}(\xi)|^2\langle \xi\rangle^{2s} d\xi &\leq C \sum_{n\in \Gamma_{1}\cap\mathbb{Z}^{d}}\langle n\rangle^{2s} \int_{\Lambda_{n}} |\widehat{\psi f}(\xi)|^2d\xi
\\
&
=C\int_{[0,1]^{d}} \sum_{n\in \Gamma_{1}\cap\mathbb{Z}^{d}}\langle n\rangle^{2s} |\widehat{\psi f}(n+t)|^2dt
\\
&
\leq C\sup_{t\in[0,1]^{d}}\sum_{n\in \Gamma_{1}\cap\mathbb{Z}^{d}}\langle n\rangle^{2s} |\widehat{e_{-t}\psi f}(n)|^2<\infty.
\end{align*}
Therefore, $(x_{0},\xi_{0})\notin WF(f)$.

$(ii)\Rightarrow (i)$. A variant of the argument employed above, but with integrals instead of sums, applies to show that $(x_0,\xi_{0})\notin WF(f)$ implies the following property: There exist an open cone $\Gamma$ and  $\varepsilon\in(0,1)$ such that for every bounded set $B\subset\mathscr{D}(I_{\varepsilon,x_{0}})$ 
			\begin{equation}\label{sob4}
					\sup_{\psi\in B}\int_{\Gamma}|\widehat{f\psi}(\xi)|^2\langle \xi\rangle^{2s} d\xi<\infty.
			\end{equation}
We leave such details to the reader. So, assume that (\ref{sob4}) holds. 
Let $\Gamma_{1}$ be an open cone contaning $\xi_{0}$ such that  $\overline{\Gamma}_{1}\subset \Gamma \cup \{0\}$. Then, there is $r>0$ such that $(\Gamma_{1}+[0,1]^{d})\cap\{\xi\in\mathbb{R}^{d}:\: |\xi|\geq r\}\subset \Gamma.$ Let $\phi\in\mathscr{D}(I_{\varepsilon,x_{0}})$ such that $\phi\equiv1$ in a neighborhood of $x_{0}$. Consider a measurable function $t:\Gamma\to [0,1]^{d}$. Taking the bounded set $B=\{\psi_{i,t}\in \mathscr{D}(I_{\varepsilon,x_{0}}): \: \psi_{j,t}(x)=x_{j}e^{-2\pi i x\cdot t(\xi)}\phi(x),\: \xi\in \Gamma\, , j=1,\dots, d\}$ in (\ref{sob4}), we obtain that there is a constant $C>0$ such that 
\begin{equation}\label{sob5}
\int_{\Gamma}|\nabla(\widehat{\phi f})(\xi+t(\xi))|^2\langle \xi\rangle^{2s} d\xi< C.
\end{equation}
The constant $C$ is actually independent of $t$.
For each $n\in \Gamma_{1}\cap \mathbb{Z}^{d}$, let $\Lambda_n$ be the unite cube $\prod_{j=1}^d[n_{j},n_{j}+1]= n+[0,1]^{d}$. Note that $\Lambda_{n}\subset \Gamma$ if $|n|\geq r$.
Then 
$$\left( \sum_{n\in \Gamma_{1}\cap \mathbb{Z}^{d}}^\infty|\widehat{\phi f}(n)|^2\langle n \rangle^{2s}\right)^{1/2}=\left(\sum_{n\in \Gamma_{1}\cap \mathbb{Z}^{d}}^\infty\int_{\Lambda_n}{|\widehat{\phi f}(n)|^2\langle n \rangle^{2s}}d\xi\right)^{1/2}\leq I_{1}^{1/2}+I_{2}^{1/2},
$$
where $I_{1}:=\sum_{n\in \Gamma_{1}\cap \mathbb{Z}^{d}}^\infty\int_{\Lambda_n}{|\widehat{\phi f}(n)-\widehat{\phi f}(\xi)|^2\langle n \rangle^{2s}} d\xi$ and 
\begin{align*}I_2&:=
\sum_{n\in \Gamma_{1}\cap \mathbb{Z}^{d}}^\infty\int_{\Lambda_n}|\widehat{\phi f}(\xi)|^2\langle n \rangle^{2s}d\xi 
\\
&
\leq     
\sum_{|n|\leq r}^\infty\int_{\Lambda_n}|\widehat{\phi f}(\xi)|^2\langle n \rangle^{2s}d\xi +
C' \int_\Gamma{|\widehat{\phi f}(\xi)|^2\langle \xi\rangle^{2s}}d\xi<\infty.
\end{align*}
It remains to show that $I_{1}$ is finite. Given $\theta>0$, define $t_{\theta}:\Gamma\to [0,1]^{d}$ as
$$
t_{\theta}(\xi)=
\begin{cases}
\theta (n-\xi) & \mbox{ if }\xi\in\Lambda_{n} \mbox{ and } |\xi|\geq r, \\
0 & \mbox{otherwise} .
\end{cases} 
$$
We now make use of (\ref{sob5}). Since
$$|\widehat{\phi f}(\xi)-\widehat{\phi f}(n)|^{2}\leq |n-\xi| \int_{0}^{1}|\nabla(\widehat{\phi f})(\xi+\theta (n-\xi))|^{2}d\theta,
$$ 
we have
\begin{align*}
I_1 \leq &
\underset{n\in \Gamma_{1}\cap \mathbb{Z}^{d}}{\sum_{|n|\leq r}^\infty}\int_{\Lambda_n}{|\widehat{\phi f}(n)-\widehat{\phi f}(\xi)|^2\langle n \rangle^{2s}} d\xi
\\
&+ C'\sup_{\theta\in[0,1]}
\int_\Gamma{|\nabla(\widehat{\phi f})(\xi+t_{\theta}(\xi))|^2\langle \xi\rangle^{2s}}d\xi<\infty.
\end{align*}
This completes the proof.
\end{proof}


\begin{thebibliography}{99}

\bibitem{AMS} P. Antosik, J. Mikusi\' nski, R. Sikorski, 
\emph{Theory of distributions. The sequential approach},
 Elsevier Scientific Publishing Co., Amsterdam; PWN---Polish Scientific Publishers, Warsaw, 1973.

\bibitem{Beals} R. Beals, \emph{Advanced mathematical analysis. Periodic functions and distributions, complex analysis, Laplace transform and applications.} Graduate Texts in Mathematics, No. 12, Springer-Verlag, New York-Heidelberg, 1973.

\bibitem{CJT1} S. Coriasco, K. Johanson, J. Toft, \emph{Local wave-front sets of Banach and Fr\'{e}chet types, and pseudo-differential operators}, Monatsh. Math. 169 (2013), 285--316.

\bibitem{CJT2} S. Coriasco, K. Johanson, J. Toft, \emph{Global wave-front sets of Banach, Fr\'{e}chet and modulation space types, and pseudo-differential operators},
 J. Differential Equations 254 (2013),  3228--3258.
 
\bibitem{Hor} L. H\"{o}rmander, \emph{The analysis of linear partial differential operators I: Distribution theory and Fourier analysis}, Springer-Verlag, 1983.

\bibitem{Hor2} L. H\"{o}rmander, \emph{Lectures on nonlinear hyperbolic differential equations}, Springer-Verlag, 1997.

\bibitem{JPTTo} K. Johansson, S. Pilipovi\'{c}, N. Teofanov, J. Toft, \emph{Gabor pairs, and a discrete approach to wave-front sets}, Monatsh. Math. 166 (2012), 181--199.

\bibitem{Kanwal} R. P. Kanwal, \emph{Generalized functions. Theory and applications}, Birkh\"{a}user Boston, Inc., Boston, MA, 2004.

\bibitem{pnt1} S. Pilipovi\'{c}, N. Teofanov, J. Toft, \emph{Micro-local analysis with Fourier-Lebesgue spaces. Part I},  J. Fourier Anal. Appl. 17 (2011), 374--407.

\bibitem{RT2007} M. Ruzhansky, V. Turunen, \emph{On the Fourier analysis of operators on the torus,} in: Modern trends in pseudo-differential operators, 87--105, Oper. Theory Adv. Appl. 172 (2007), Birkh\"{a}user, Basel.

\bibitem{RT2009} M. Ruzhansky, V. Turunen, \emph{On the toroidal quantization of periodic pseudo-differential operators,} Numer. Funct. Anal. Optim. 30 (2009), 1098--1124.

\bibitem{RT} M. Ruzhansky, V. Turunen,
\emph{ Quantization of pseudo-differential operators on the torus}, J. Fourier Anal. Appl. 16 (2010), 943--982.

\bibitem{Ruzh} M. Ruzhansky, V. Turunen, \emph{Pseudo-differential operators and symmetries. Background analysis and advanced topics}, Birkh\" auser Verlag, Basel, 2010.

\bibitem{Schwartz} L. Schwartz, \emph{Th\' eorie des distributions}, Hermann, Paris, 1966.

\bibitem{Skornik} K. Sk\' ornik, \emph{Hereditarily periodic distributions}, Studia Math. 43 (1972), 245--272. 
 
\bibitem{Vindas} J. Vindas, R. Estrada, \emph{Distributional point values and convergence of Fourier series and integrals}, J. Fourier Anal. Appl. 13 (2007), 551--576. 

\bibitem{Walter}G. Walter, \emph{Pointwise convergence of distribution expansions}, Studia Math. 26 (1966), 143--154.

\end{thebibliography}
\end{document}